\documentclass{amsart}

\usepackage{amssymb}
\usepackage{amsmath}
\usepackage{amsthm}
\usepackage{amsbsy}
\usepackage{graphicx}
\usepackage{bm}

\theoremstyle{plain}
\newtheorem{theorem}{Theorem}[section]
\newtheorem{lemma}[theorem]{Lemma}

\theoremstyle{definition}
\newtheorem{definition}[theorem]{Definition}

\theoremstyle{remark}

\newtheorem*{acknowledgements}{Acknowledgements}

\newcommand{\F}{\hat{F}}

\newcommand{\del}{\partial}

\begin{document} 

\title{Incompressibility and Least-Area surfaces}

\author{Siddhartha Gadgil}

\address{	Department of Mathematics\\
		Indian Institute of Science\\
		Bangalore 560003, India}

\email{gadgil@math.iisc.ernet.in}

\date{\today}

\subjclass{Primary 57N10; Secondary 53A10}

\begin{abstract}

We show that if $F$ is a smooth, closed, orientable surface embedded
in a closed, orientable $3$-manifold $M$ such that for each Riemannian
metric $g$ on $M$, $F$ is isotopic to a least-area surface $F(g)$,
then $F$ is incompressible.

\end{abstract}

\maketitle

\section{Introduction}

We assume throughout that all manifolds (and surfaces) we consider are
orientable.  Let $M$ be a closed, smooth, $3$-manifold and let $F$ be
a smoothly embedded surface in $M$.

For a Riemannian metric $g$ on $M$, we can seek to minimise the area
of embedded surfaces in the homotopy class of $F$. Here, we say two
embedded surfaces $F$ and $F'$ in $M$ are homotopic if there is a
homeomorphism $\varphi:F\to F'$, so that if $i_F:F\to M$ and
$i_{F'}:F'\to M$ denote the inclusion maps, then $i_F$ is homotopic to
$i_{F}'\circ \varphi:F\to M$. We consider the functional
$$A(g,F)=inf\{Area_g(F'):\text{$F'$ embedded surface homotopic to
$F$}\}$$

\begin{definition}
A surface $F$ is said to be \emph{least area} with respect to the
metric $g$ if $Area_g(F)=A(g,F)$
\end{definition}

We recall the concept of \emph{incompressibility} of the surface $F$.
\begin{definition}
A closed, smoothly embedded surface $F\subset M$ in a $3$-manifold $M$
is said to be incompressible if the following conditions hold:
\begin{itemize}
\item If $D\subset M$ is a smoothly embedded $2$-disc with $\del
  D\subset F$ and $int(D)\cap F=\phi$, then there is a disc $E\subset
  F$ with $\del E=\del D$.
\item If $F$ is a $2$-sphere, then $F$ does not bound a $3$-ball.
\end{itemize}
\end{definition}

A disc $D\subset M$ such that $\del D\subset F$ and $int(D)\cap
F=\phi$ so that $D$ is transversal to $F$ is called a
\emph{compressing disc} if there is no disc $E\subset F$ with $\del
E=\del D$.

If $F$ is \emph{incompressible} and $M$ is \emph{irreducible}, then a
fundamental result (\cite{MY1}, see also~\cite{HS}) is that for each
metric $g$, the above infimum is attained for some smooth, embedded
surface $F(g)$, i.e., there is a least area surface. Recall that $M$
is said to be irreducible if every embedded $2$-sphere $S\subset M$
bounds a $3$-ball in $M$. Least area surfaces enjoys several very
useful properties~\cite{FHS} -- for instance, leading to the
equivariant Dehn's lemma of Meeks and Yau~\cite{MY2}.

We show here that, conversely, the property of having least area
representatives for each Riemannian metric characterises
incompressibility. Fix henceforth a closed, smooth, orientable surface
$F$ embedded in a closed, orientable, $3$-manifold $M$.

\begin{theorem}
Suppose for each Riemannian metric $g$ on $M$, there is a smooth,
embedded surface $F(g)$ homotopic to $F$ such that
$Area_g(F(g))=A(g,F)$, then $F$ is incompressible.
\end{theorem}

Our result gives a geometric characterisation of incompressiblity,
which may be useful in proving incompressibility for surfaces
constructed as limits. In particular, this result was motivated by Tao
Li's proof~\cite{Li} of the Waldhausen conjecture in the non-Haken
case, where an incompressible surface was constructed as a limit of
strongly irreducible Heegaard splittings.

Henceforth assume, without loss of generality, that $F$ is
connected. If $F$ is a $2$-sphere, either $F$ is incompressible or $F$
bounds a $3$-ball in $M$. In the first case, there is nothing to
prove. In the second case, $F$ is homotopically trivial and hence
homotopic to the boundary of any $3$-ball in $M$. By considering the
boundaries of arbitrarily small balls, we see that $A(g,F)=0$ for any
metric $g$. Thus there is no embedded surface $F(g)$ with
$Area_g(F(g))=A(g,F)$. Thus, we can assume henceforth that $F$ is not
a $2$-sphere.

Suppose $F$ is not incompressible (and $F$ is not a $2$-sphere), then
there is a compressing disc $D$ for $F$. A regular neighbourhood of
$F\cup D$ has two boundary components, one of which is parallel to
$F$. Denote the other by $F'$. We call $F'$ the result of compressing
$F$ along $D$. Observe that $F$ is obtained from $F'$ by adding a
$1$-handle. If $F$ is a $2$-sphere that bounds a $3$-ball, we declare
the empty set to be the result of compressing $F$. 

Given any surface $F$, we can inductively define a sequence of
compressions. Namely, if $F$ is not incompressible, then we compress
$F$ along some compressing disc $D$ to get $F'$. We repeat this
process for each component of $F'$ which is not compressible. As the
maximum of the genus of the components of the surface $F'$ obtained by
compressing $F$ is less than the genus of $F$, this process terminates
after finitely many steps. The result is a (possibly empty) surface
$\F$, each component of which is incompressible.

Suppose the result of the compressions is empty, then $F$ is homotopic
to the boundary of a handlebody, i.e., the boundary of a regular
neighbourhood of a graph $\Gamma\subset M$. By considering arbitrarily
small neighbourhoods of $\Gamma$, we see that for any metric $g$,
$A(g,F)=0$. Thus there is no embedded surface $F(g)$ with
$Area_g(F(g))=A(g,F)$, i.e., the hypothesis of the theorem cannot be
satisfied.

Thus, we can, and do, assume henceforth that $F$ is not homotopic to
the boundary of a handlebody and $\hat F$ is not empty. Then $F$ is
obtained from $\hat F$ by addition of $1$-handles. Given $\epsilon>0$,
the $1$-handles can be attached to $\F$ so that the area of the
resulting surface, which is homotopic to $F$, is at most
$Area_g(\F)+\epsilon$. It follows that
$$A(g,F)\leq Area_g(\F)$$

We construct a metric $g$ which is a warped product in a neighbourhood
$N(\F)$ of $\F$ so that any least-area surface not contained in
$N(\F)$ has area greater than the area of $\F$. Thus, if a least-area
surface homotopic to $F$ exists, it must be contained in $N(\F)$. The
structure of the metric on $N(\F)$ together with some topological
arguments show that this cannot happen unless $F=\F$, i.e., $F$ is
incompressible.

\section{Construction of the metric}

In this section we construct the desired metric for which $F$ has no
least-area representative unless $F=\F$.

As $\F$ and $M$ are orientable, a regular neighbourhood $N(\F)$ of $\F$
is a product. We shall identify this with $\F\times [-T,T]$, with $T$
to be specified later. We shall also consider the regular
neighbourhood $n(\F)=\F\times [-1,1]\subset N(\F)$.

Choose and fix a metric of constant curvature $1$, $0$ or $-1$ on each
component of $\F$ and denote this $g_0$. Let the area of $\F$ with
respect to $g_0$ be $A_0$. 

We shall use the \emph{monotonicity lemma} of Geometric measure
theory. We state this below in the form we need. For a stronger result
in the Riemannian case, see~\cite{Fr}.

\begin{lemma}[Monotonicity lemma]
There exist constants $\epsilon>0$ and $R>0$ such that if $g$ is a
Riemannian metric on $M$ and $x$ is a point so that the sectional
curvature of $g$ on the ball $B_g(x,R)$ of radius $R$ around $x$ (with
respect to $g$) has sectional curvature satisfying $\vert K\vert\leq
\epsilon$ and $F$ is a least-area surface with $x\in F$, then
$Area_g(F\cap B(x,R))>A_0.$
\end{lemma}

We shall construct the desired metric in the following lemma.

\begin{lemma}
There is a Riemannian metric $g$ on $M$ satisfying the following
properties.
\begin{itemize}
\item On $N(\F)$, $g$ is of the form $g = f(t)g_0\oplus dt^2$, with
$f$ a smooth function with $f(0)=1$ and $f(t)>1$ for $t\neq 0$.
\item For $x\in M-int(N(\F))$, the sectional curvature of $g$ on $M$
  satisfies $\vert K\vert\leq \epsilon$.
\item For $x\in M-int(N\F))$, the injectivity radius at $x$ is greater
  than $R$.
\end{itemize}
\end{lemma}
\begin{proof}
Observe that $N(\F)-int(n(\F))$ has two components for each component
$\F_0$ of $\F$, each of which can be identified with $\F_0\times
[0,1]$ with $\F\times \{1\}$ a component of $\del N(\F)$ and $\F\times
\{0\}$ a component of $\del n(\F)$. On each such component consider
the product Riemannian metric $g_0\oplus dt^2$. Extend this smoothly
to a metric on the complement $M-int(n(\F))$ of the interior of
$n(\F)$. Rescale the metric by a constant $s>1$ to ensure that it has
sectional curvature satisfying $\vert K\vert\leq \epsilon$ and the
injectivity radius at each point outside $int(N(\F))$ is at least
$R$. We choose the constant $s$ to be greater than $1$ even if this is
not necessary to ensure the bounds on curvature and the injectivity
radius. We denote the rescaled metric, defined on $M-int(n(\F))$, by
$g$. The restriction of $g$ to each component of $N(T)$ can be
identified with the product metric $sg_0+dt^2$.

Let $T=1+s$. Then there is a natural identification of $N(\F)$ with
$\F\times [-T,T]$, with $n(\F)$ identified with $\F\times [-1,1]$ and
with the restriction of the metric $g$ to $N(\F)-int(n(\F))$ given by
$s g_0\oplus dt^2$.

We extend the constant function $f(t)=s$ on $[-T,-1]\cup [1,T]$ to a
smooth function on $[-T,T]$ with $f(0)=1$ and $f(t)>1$ if $t\neq
0$. The Riemannian metric on the complement of $n(\F)$ extends
smoothly to one given by $g=f(t)g_0\oplus dt^2$ on $N(\F)$. This
satisfies all the conditions of the lemma.
\end{proof}

Note that by construction $\F\times\{0\}$ is isometric to $\F$ with
the metric $g_0$. Further the projection map $p:F\times [-T,T]\to F$
is (weakly) distance decreasing, and strictly distance decreasing
outside $\F\times\{0\}$.

\section{Proof of incompressibility}

Suppose now that there is a surface $F(g)$ homotopic to $F$ with area
$A(g,F)$.

\begin{lemma}
We have $F(g)\subset \F_0\times (-T,T)$ for some component $\F_0$ of
$\F$.
\end{lemma}
\begin{proof}
If there is a point $x\in F(g)-\F\times (-T,T)$, the monotonicity
lemma applied to $F(g)\cap B(x,R)$ shows that the area of $F(g)$ is
greater than $A_0$, a contradiction. Further, as $F(g)$ is connected,
for some component $\F_0$ of $\F$, $F(g)\subset \F_0\times [-T,T]$.
\end{proof}

To simplify notation, we henceforth denote the surface $F(g)$ by
$F$. We consider the restriction of the projection map from $F_0\times
[-T,T]$ to $F$, which we denote, by abuse of notation, by $p:F\to
\F_0$. We have seen that this is strictly distance decreasing unless
$F=\F_0$.

We have two cases, depending on whether the embedded surface $F\subset
\F_0\times (-T,T)$ separates the boundary components of $\F_0\times
[-T,T]$. 

In case $F$ does not separate the boundary components of $\F_0\times
[-T,T]$, there is a curve $\gamma$ joining the boundary components
disjoint from $F$. By considering cup products, it follows that
$p:F\to \F_0$ has degree zero. On the other hand, if $F$ does separate
the boundary components of $\F_0\times [-T,T]$, as $F$ is connected
there is a curve $\gamma$ joining the boundary components intersecting
$F$ transversely in one point. It follows that we can choose an
orientation on $F$ so that $p:F\to \F_0$ has degree one.

Recall that $F$ is a connected, orientable surface that is not a
$2$-sphere. As $\F_0\times (-T,T)$ deformation retracts to $\F_0$, the
homotopy class of the inclusion map is determined by the homotopy
class of $p$. If $\F_0$ is not a $2$-sphere, then the homotopy class
of $p$ is determined by the induced map on fundamental groups. If
$\F_0$ is a $2$-sphere, then the homotopy class of $p$ is determined
by the degree of $p$.

We show first that the case where $F$ does not separate the boundary
components of $\F_0\times [-T,T]$ cannot occur.

\begin{lemma}
The surface $F$ must separate the boundary components of $\F_0\times
[-T,T]$.
\end{lemma}
\begin{proof}
We have seen that $p:F\to \F_0$ has degree zero. Suppose first that
$\F_0$ is not a $2$-sphere. By a theorem of Hopf and
Knesser~\cite{Ho1}\cite{Ho2}\cite{Kn1}\cite{Kn2}, it follows that $p$
is homotopic to a map whose image does not contain some point $p\in
F_0$. It follows that $G=p_*(\pi_1(F))$ (which is finitely-generated)
is conjugate to the subgroup of a free group and hence is a finitely
generated free group. The projection $p:\F_0\times [-T,T]\to \F_0$
induces an isomorphism of fundamental groups. This gives an
identification of $G$ with the image of $\pi_1(F)$ under the
homomorphism induced by the inclusion $i:F\to \F_0\times [-T,T]$.

By a theorem of Jaco~\cite{Ja}, there is a finite graph $\Gamma$ and
$\pi_1(\Gamma)$ isomorphic to $G$ and a map $f:F\to\Gamma$ so that the
mapping cylinder $M(f)$ of $f$ is a handlebody. Moreover, we have an
identification of $\pi_1(\Gamma)$ with $G$ with respect to which the
homomorphism from $\pi_1(F)$ to $G\subset\pi_1(\F_0\times [-T,T])$
induced by the inclusion corresponds to the map induced by inclusion
from $\pi_1(F)$ to  $\pi_1(M(f))=\pi_1(\Gamma)$.

Choose an embedding of $\Gamma$ in $\F_0\times [-T,T]$ with induced
map on fundamental groups $\pi_1(\Gamma)\to G\subset\pi_1(\F_0\times
[-T,T])$ corresponding to the above identification of $\pi_1(\Gamma)$
with $G$. Then the surface $F$ is homotopic to the boundary of a
regular neighbourhood of the image of $\Gamma$, which is a handlebody
in $\F_0\times [-T,T]$. We have seen that in this case the hypothesis
of the theorem cannot be satisfied.

Finally consider the case when $\F_0$ is a $2$-sphere. As $p$ has
degree zero, $p$ is homotopic to a constant map. Hence the inclusion
map $i:F\to \F_0\times (-T,T)$ is also homotopic to a constant map. It
follows that $F$ is homotopic to the boundary of a handlebody. As
above, the hypothesis of the theorem cannot be satisfied in this case.

\end{proof}

It thus suffices to consider the case when $F$ does separate the
boundary components of $\F_0\times [-T,T]$.

\begin{lemma}
Suppose $F$ does separate the boundary components of $\F_0\times
[-T,T]$, then $F=\F_0$.
\end{lemma}
\begin{proof}
We use a theorem of Edmonds~\cite{Ed} regarding degree-one maps
$\varphi:F\to F'$ between closed surfaces. Namely, there is a map
$\psi$ homotopic to $\varphi$, a compact, connected subsurface
$\Sigma$ in $F$ and a disc $D\subset F'$ such that
$\psi(\Sigma)\subset D$ and $\psi$ maps $F-int(\Sigma)$
homeomorphically onto $F'-int(D)$. We can regard $F$ as obtained from
$F'$ by attaching $1$-handles to $D$. Further, all the cores and
co-cores of these $1$-handles are mapped to homotopically trivial
curves by $\psi$.

It follows that an embedded surface homotopic to $F$ is obtained from
$\F_0$ by adding $1$-handles (corresponding to those required to
obtain $F$ from $\F_0$). Thus, if $A_1$ is the area of $\F_0\times
\{0\}$, then $A(g,F)\leq A_1$, hence $Area_g(F)\leq A_1$.

Let $d\omega$ and $dA$ denote the area forms on $\F_0=\F_0\times
\{0\}$ and $F$ respectively. Then for a smooth function $h$ on $F$,
$p^*(d\omega)=hdA$. As the projection map is distance-decreasing,
$h(p)\leq 1$ for all $p$, with equality at all points only in the case
where $F=\hat F_0$.

Observe that $A_1=\int_{\F_0} d\omega=\int_{F} p^*(d\omega)=\int_{F}
hdA$, where the second equality holds as $p$ has degree one. Further,
if $F\neq \F_0$, then $\int_{F} hdA<\int_{F} dA= Area
(F)$. Thus, $A_1<Area_g(F)$, a contradiction.

It follows that $F=\F_0$.
\end{proof}

Thus, the surface $F$ must be homotopic to $\F_0$, which is
incompressible. Recall that we can assume that $F$ is not a
$2$-sphere. It follows that $F$ is incompressible by the
characterisation of incompressible surfaces as those for which the
induced map on fundamental groups is injective (if $F$ is not the
$2$-sphere). This contradicts our assumption that $F$ is compressible,
completing the proof of the theorem.\qed

\begin{acknowledgements}
I thank the referee for helpful comments and for pointing out that the
hypothesis of irreducibility in an earlier version was unnecessary. I
thank Harish Seshadri for helpful discussions.
\end{acknowledgements}

\bibliographystyle{amsplain}

\end{document}